\documentclass{article}

\date{}

\usepackage{amsmath,amssymb,amsthm}
\usepackage{color}

\usepackage[latin1]{inputenc}
\usepackage{subfigure}
\usepackage{amsmath,amssymb}

\newtheorem{rem}{Remark}

\newtheorem{conjecture}{Conjecture}

\newtheorem{example}{Example}
\newtheorem{proposition}{Proposition}
\newtheorem{theorem}{Theorem}
\newtheorem{corollary}{Corollary}

\def\F{\mathbb F}
\def\A{\mathcal A}
\def\C{\mathbb C}
\def\H{\mathcal H}
\def\Z{\mathbb Z}

\def\O{\mathcal O}
\def\W{\mathcal W}

\author{
Maki Nakasuji
and 
Hiroshi Naruse
}
\title
{Yang-Baxter basis of Hecke algebra and Casselman's problem
(extended abstract)}

\begin{document}
\maketitle


\begin{abstract}

We generalize the definition of   Yang-Baxter basis of type $A$ Hecke algebra introduced by
A.Lascoux, B.Leclerc and J.Y.Thibon (Letters in Math. Phys., 40 (1997), 75--90)
to all the Lie types and prove their duality.
As an application we give a solution to Casselman's problem on Iwahori fixed vectors
of principal series representation of $p$-adic groups.

\end{abstract}

\section{Introduction}

 Yang-Baxter basis of Hecke algebra of type $A$ was defined in the paper of 
 Lascoux-Leclerc-Thibon
 \cite{LLT}. There is also a modified version in \cite{Las}.
 First 
 we generalize the latter version to all the Lie types.
Then we will solve the Casselman's problem on the basis of Iwahori fixed vectors 
using
Yang-Baxter basis and Demazure-Lusztig type operator.
This paper is an extended abstract and the detailed proofs will appear in \cite{NN}.

\section{Generic Hecke algebra}

\subsection{Root system, Weyl group and generic Hecke algebra}

Let ${ \cal R}=(\Lambda, \Lambda^{*},R,R^{*})$ be a (reduced) semisimple root data cf. \cite{Dem}.
More precisely 
$\Lambda\simeq \Z^r$ is a  weight lattice with $\text{ rank }\Lambda=r$.
There is a pairing 
$<\;,\;>: \Lambda^{*}\times \Lambda\to \Z$.
$R\subset \Lambda$ is a root  system with simple roots $\{\alpha_i\}_{1\leq i\leq r}$ and 
positive roots $R^{+}$.
$R^{*}\subset \Lambda^{*}$ is the set of coroots, and
there is a bijection $R\to R^{*}$, 
$\alpha\mapsto \alpha^{*}$. We also denote the coroot $\alpha^{*}=h_\alpha$.
The Weyl group $W$ of $\cal R$ is generated by simple reflections $S=\{s_i\}_{1\leq i\leq r}$.
The action of
$W$ on $\Lambda$ is given by $s_{i}(\lambda)=\lambda-
<\alpha^{*}_{i},\lambda>\alpha_{i}$ for $\lambda\in \Lambda$. 
We define generic Hecke algebra
$H_{t_1,t_2}(W)$  over $\Z[t_1,t_2]$ with two parameters $t_1,t_2$ as follows.
Generators are  $h_i=h_{s_i}$, 
with relations $(h_i-t_1)(h_i-t_2)=0$ for $1\leq i\leq r$ and  the braid relations
$\underbrace{h_i h_j\cdots}_{m_{i,j}}=\underbrace{h_j h_i\cdots}_{m_{i,j}}$
, where $m_{i,j}$ is the order of $s_i s_j$ for $1\leq i<j\leq r$.
We need to extend the coefficients to the quotient field of the group algebra $\Z[\Lambda]$. An element of 
$\Z[\Lambda]$ is denoted as $\displaystyle\sum_{\lambda\in \Lambda} c_\lambda e^\lambda$.
The Weyl group acts on $\Z[\Lambda]$ by $w (e^\lambda)=e^{w\lambda}$.
We extend the coefficient ring $\Z[t_1,t_2]$ of $H_{t_1,t_2}(W)$ to
$$Q_{t_1,t_2}(\Lambda):=\Z[t_1,t_2]\otimes Q(\Z[\Lambda])$$
where $Q(\Z[\Lambda])$ is the quotient field of $\Z[\Lambda]$.
$$H^{Q(\Lambda)}_{t_1,t_2}(W):=Q_{t_1,t_2}(\Lambda)\otimes_{\Z[t_1,t_2]} H_{t_1,t_2}(W).$$
For $w\in W$, an expression  of $w=s_{i_1} s_{i_2}\cdots s_{i_\ell}$ with minimal number of  generators $s_{i_k}\in S$ is called  a reduced
expression in which case we write $\ell(w)=\ell$ and call it  the length of $w$.
Then $h_w=h_{i_1}h_{i_2}\cdots h_{i_\ell}$ is well defined and
$\{h_w \}_{w\in W}$ forms a $Q_{t_1,t_2}(\Lambda)$-basis of $H^{Q(\Lambda)}_{t_1,t_2}(W)$.

\subsection{Yang-Baxter basis and its properties}

Yang-Baxter basis was introduced in the paper
\cite{LLT}
to investigate the relation with
Schubert calculus.
There is also a variant in \cite{Las} for type $A$ case.
We generalize that results to all Lie types.

For $\lambda\in \Lambda$, we define $E(\lambda)=e^{-\lambda}-1$. Then
$E(\lambda+\nu)=E(\lambda)+E(\nu)+E(\lambda)E(\nu)$. In particuar,  if $\lambda\neq 0$,
$\frac{1}{E(\lambda)}+\frac{1}{E(-\lambda)}=-1$.
\begin{proposition} For $\lambda\in \Lambda$, if $\lambda\neq 0$, 
let $h_i(\lambda):=h_i+\frac{t_1+t_2}{E(\lambda)}$.
Then these satisfy the {\bf Yang-Baxter relations}, i.e. 
 if we write $[p,q]:=p\lambda+q\nu$ for fixed 
$\lambda,\nu\in \Lambda$,
the following equations hold.
We assume all appearance of $[p,q] $ is nonzero. 
$$
\begin{array}{lcll}
h_i([1,0]) h_j([0,1])&=&h_j([0,1]) h_i([1,0]) & \text{ if }\; m_{i,j}=2\\[0.1cm]
h_i([1,0]) h_j([1,1])h_i([0,1])&=&
h_j([0,1]) h_i([1,1]) h_j([1,0]) & \text{ if }\; m_{i,j}=3\\[0.1cm]
h_i([1,0]) h_j([1,1])h_i([1,2])h_j([0,1])&=&
h_j([0,1]) h_i([1,2]) h_j([1,1]) h_i([1,0]) & \text{ if }\; m_{i,j}=4\\[0.2cm]
h_i([1,0]) h_j([1,1])h_i([2,3]) 
&&
h_j([0,1]) h_i([1,3]) h_j([1,2]) \\[0.1cm]
\hspace{1cm}\times h_j([1,2]) h_i([1,3]) h_j ([0,1])&=&\hspace{1cm}\times h_i([2,3]) h_j([1,1]) h_i([1,0])
 & \text{ if }\; m_{i,j}=6\\
\end{array}
$$
\end{proposition}
\begin{proof}
We can prove these equations by direct calculations.
\end{proof}

\begin{rem}
In \cite{Che} I. Cherednik 
treated Yang-Baxter relation in more general setting. 
There is also a related work \cite{Kat} by S. Kato and the proof of Theorem 2.4
in \cite{Kat} suggests a uniform way to prove Yang-Baxter relations
without direct calculations.
\end{rem}

We use the Bruhat order $x\leq y$ on elements  $x,y\in W$ (cf.\cite{Hum}).
Following \cite{Las} we
define the Yang-Baxter basis $Y_w$ for $w\in W$ recursively as follows. 
\begin{center}
$Y_e:=1$, $Y_{w}:=Y_{w'}(h_i+\frac{t_1+t_2}{w' E(\alpha_i)})$
if $w=w' s_i>w'$.
\end{center}
Using the Yang-Baxter relation above
 it is easy to see that $Y_w$ does not depend on a reduced expression of $w$.
As the leading term of $Y_w$ with respect to the Bruhat order is $h_w$,  they also form 
a $Q_{t_1,t_2}(\Lambda)$-basis $\{Y_w\}_{w\in W}$ of $H^{Q(\Lambda)}_{t_1,t_2}(W)$.
We are interested in the transition coefficients 
$p(w,v)\text{ and }\tilde{p}(w,v)\in Q_{t_1,t_2}(\Lambda)$
between the two basis
$\{Y_w\}_{w\in W}$ and $\{h_w\}_{w\in W}$ , i.e. 

$$Y_v=\sum_{w\leq v} p(w,v) h_w, \text{ and } h_v=\sum_{w\leq v} \tilde{p}(w,v) Y_w.$$

Take a reduced expression of $v$ e.g. $v=s_{i_1}\cdots s_{i_\ell}$ where $\ell=\ell(v)$ is the length of $v$
(cf. \cite{Hum}).
Then $Y_v$ is expressed as follows.
$$Y_{v}=\prod_{j=1}^{\ell}\left(h_{i_j}+\frac{t_1+t_2}{E(\beta_j)}\right)$$
where
$\beta_j:=s_{i_1}\cdots s_{i_{j-1}}(\alpha_{i_j})$ for $j=1,\ldots,\ell$.
The set $R(v):=\{\beta_1,\ldots, \beta_\ell\}\subset R^{+}$ is independent of the reduced expression of $v$.
The Yang-Baxter basis defined in \cite{LLT} is normalized as follows.
$${Y}^{LLT}_{v}:=\left(\prod_{j=1}^{\ell} \frac{E(\beta_j)}{t_1+t_2}\right) Y_{v}=
\prod_{j=1}^{\ell}\left(\frac{E(\beta_j)}{t_1+t_2}h_{i_j}+1\right).$$

\begin{rem}
The relation to $K$-theory Schubert calculus is as follows.
If we set $t_1=0,t_2=-1$
and replacing $\alpha_i$ by $-\alpha_i$.
Then the coefficient of $h_w$ in ${Y}^{LLT}_v$  is the localization $\psi^w(v)$
at $v$ of the equivariant $K$-theory Schubert class $\psi^w$ (cf. \cite{LSS}).
\end{rem}

\noindent
Let $w_0$ be the longest element in $W$.
Define 
$Q_{t_1,t_2}(\Lambda)$-algebra homomorphism $\Omega:H^{Q(\Lambda)}_{t_1,t_2}\to
 H^{Q(\Lambda)}_{t_1,t_2}$ by $\Omega(h_{w})=h_{w_0 w w_0}$.
Let $\star$ be the ring homomorphism on $\Z[\Lambda]$ induced by $\star(e^{\lambda})=e^{-\lambda}$ and
extend to $Q_{t_1,t_2}(\Lambda)$.

\begin{proposition}
(Lascoux \cite{Las} Lemma 1.8.1  for type $A$ case)\;
For $v\in W$,
$$\Omega(Y_{w_0 v w_0})=\star[{w_0}( Y_{v})]
$$
where $W$ acts only on the coefficients.
\end{proposition}

\begin{proof}
When $\ell(v)>0$ there exists $s\in S$ such that $v=v' s>v'$.
Using the induction assumption on $v'$, we get the formula for $v$.
\end{proof}

Taking the coefficient of $h_w$ in  the above equation, we get
\begin{corollary}
$$p(w_0 w w_0,w_0 v w_0)=\star[w_0 p(w,v)].$$
\end{corollary}

\subsection{Inner product and orthogonality}

Define inner product $(\;,\;)^H$ on $H^{Q(\Lambda)}_{t_1,t_2}(W)$
by $(f,g)^H:=\text{ the coefficient of } h_{w_0}$ in $f g^\vee$, where
$g^\vee=\sum_{} c_w h_{w^{-1}}$ if 
$g=\sum_{} c_w h_{w}$.
It is easy to see that $(f h_s, g)^H=(f,g h_s)^H$ for 
$f,g\in H^{Q(\Lambda)}_{t_1,t_2}(W)$ and $s\in S$.
There is an involution $\hat{}:H^{Q(\Lambda)}_{t_1,t_2}\to H^{Q(\Lambda)}_{t_1,t_2}$ defined by 
$\hat{h}_i=h_i-(t_1+t_2), \hat{t}_1=-t_2, \hat{t}_2=-t_1$.
It is easy to see that $\hat{h}_{s} h_{s}= -t_1 t_2$ for $s\in S$.

The following  proposition is due to A.Lascoux for the type $A$ case  \cite{Las}  P.33.

\begin{proposition}
For all $v,w\in W$,
$$(h_v, \hat{h}_{w_0 w})^H=\delta_{v,w}.
$$
\end{proposition}
\begin{proof}
We can use induction on the length $\ell(v)$ of $v$ to prove the equation.

\end{proof}

We have another orthogonality between $Y_v$ and $w_0(Y_{w_0 w})$.
\begin{proposition}(Type $A$ case was due to \cite{LLT} Theorem 5.1 , \cite{Las} Theorem 1.8.4.)
\\
For all $v,w\in W$,
$$(Y_v,  {w_0}(Y_{w_0 w}))^H=\delta_{v,w}.$$
\end{proposition}

\begin{proof}
We use induction on $\ell(v)$ and
use the fact that if $s\in S$ and $u\in W$, then
$Y_{u} h_s= a Y_{us}+b Y_{s}$ for some $a,b\in Q_{t_1,t_2}(\Lambda)$.

\end{proof}

 \subsection{Duality between the transition coefficients}
 
 Recall that we have two transition coefficients 
 $ {p}(w,v) , \tilde{p}(w,v)\in Q_{t_1,t_2}(\Lambda)$ defined by the following expansions.

$$Y_v=\sum_{w\leq v} {p}(w,v) h_w$$
$$h_v=\sum_{w\leq v} \tilde{p}(w,v) Y_w$$
 
Below gives a relation between them.
\begin{theorem}(Lascoux \cite{Las} Corollary 1.8.5 for type $A$ case) For $w,v\in W$,
$$\tilde{p}(w,v)=(-1)^{\ell(v)-\ell(w)}p(v w_0,w w_0).$$
\end{theorem}

\begin{proof}
We will calculate $(h_v, {w_0} (Y_{w_0 w}))^H$ in two ways.
As $h_v=\displaystyle\sum_{w\leq v}\tilde{p}(w,v) Y_w$,
$$(h_v, {w_0} ( Y_{w_0 w}))^H=\tilde{p}(w, v)$$
by the orthogonality on $Y_v$ (Proposition 4).
On the other hand, 
as $h_i+\frac{t_1+t_2}{E(\beta)}=\hat{h}_{i}-\frac{t_1+t_2}{E(-\beta)}$ for $\beta\in R$,
we can expand $Y_v$ in terms of $\hat{h}_w$  as follows.

$$Y_v=\sum_{w\leq v}(-1)^{\ell(v)-\ell(w)} \star[p(w,v)] \hat{h}_w.$$
\noindent
So we have
 $$
 {w_0}( Y_{w_0 w})=\sum_{w_0 v\leq w_0 w} (-1)^{\ell(v)-\ell(w)} w_0[\star p(w_0 v,w_0 w)] \hat{h}_{w_0 v}.
$$
\noindent
Then using the orthogonality on $h_v$ (Proposition 3) and Corollary 1,
  $$(h_v, {w_0}(Y_{w_0 w}))^H=(-1)^{\ell(v)-\ell(w)} w_0[\star p(w_0 v,w_0 w)]
  =(-1)^{\ell(v)-\ell(w)} p(v w_0 ,w w_0 ).$$
 The theorem is proved.

\end{proof}

\subsection{Recurrence relations}

 Here we give some recurrence relations on $p(w,v)$ and $\tilde{p}(w,v)$.
 
 \begin{proposition}(left ${p}$)\;
 For $w\in W$ and $s\in S$, if $sv>v$ then
$$p(w,sv)=\begin{cases} 
 \frac{t_1+t_2}{E(\alpha_s)}  s [p(w,v)] -t_1 t_2 s [p(sw,v)] & \text{ if }\; sw >w\\ 
  (t_1+t_2)(\frac{1}{E(\alpha_s)}+1)s [p(w,v)]+  s[p(sw,v)]& \text{ if }\; sw<w.\\
\end{cases}
 $$
\end{proposition}
 
 \begin{proof} By the definition we have
 $Y_{sv}=Y_{s} s[Y_{v}]$ from which we can deduce the recurrence formula.
 \end{proof}
 
 We note that by this recurrence we can identify $p(w,v)$ as a coefficient of transition between
two bases of the space of Iwahori fixed vectors cf.
Theorem 3 below.

 \begin{proposition} (right ${p}$)\;
For $w\in W$ and $s\in S$, if $vs>v$ then
 $$p(w,vs)=
 \begin{cases}
 \frac{t_1+t_2}{v E(\alpha_s)}  p(w,v)-t_1 t_2 p(ws,v)& \text{ if }\; ws>w\\
 (t_1+t_2)(\frac{1}{v E(\alpha_s)}+1)  p(w,v) + p(ws,v) & \text{ if }\; ws <w.\\
\end{cases}
 $$
 
 \end{proposition}
 \begin{proof}
We can use the equation $Y_{vs}=Y_{v} v[ Y_s ]$ and taking the coefficient of $h_w$,
we get the formula.

 \end{proof}

 \begin{proposition}(left $\tilde{p}$)\;
 For $w\in W$ and $s\in S$, if $sv>v$ then
 $$\tilde{p}(w,sv)=
 \begin{cases}
 -\frac{t_1+t_2}{E(\alpha_s)}\tilde{p}(w,v)+ 
 ({ t_2}+\frac{t_1+t_2}{E(\alpha_s)})({ t_2}+\frac{t_1+t_2}{E(-\alpha_s)} ) 
 s[\tilde{p}(sw,v)] & \text{ if } sw>w\\
 -\frac{t_1+t_2}{E(\alpha_s)}\tilde{p}(w,v) +
 s[\tilde{p}(sw,v)] & \text{ if } sw <w.\\
\end{cases}
 $$
 \end{proposition}
 \begin{proof}
We can prove the recurrence relation using Corollary 2 below.
 \end{proof}

 \begin{proposition}(right $\tilde{p}$)\;
  For $w\in W$ and $s\in S$, if $vs>v$ then
 $$\tilde{p}(w,vs)=
 \begin{cases}
 -\frac{t_1+t_2}{w E(\alpha_s)} \tilde{p}(w,v)+
  ({t_2}+\frac{t_1+t_2}{w E(\alpha_s)})({ t_2}+\frac{t_1+t_2}{w E(-\alpha_s)} ) \tilde{p}(ws,v) & \text{ if } ws>w\\
 -\frac{t_1+t_2}{w E(\alpha_s)} \tilde{p}(w,v)+\tilde{p}(ws,v) & \text{ if } ws <w.\\
\end{cases}
 $$
 \end{proposition}
 \begin{proof}
We can prove the recurrence relation using Corollary 2 below.
 \end{proof}

\section{Kostant-Kumar's twisted group algebra}

Let
$Q^{KK}_{t_1,t_2}(W):=Q_{t_1,t_2}(\Lambda)\# \Z[W]$ be 
the (generic) twisted group algebra of Kostant-Kumar.
Its element is of the form $\displaystyle\sum_{w\in W} f_w \delta_w$
for $f_w\in Q_{t_1,t_2}(\Lambda)$ and the product is
defined by
$$(\sum_{w\in W} f_w \delta_w)(\sum_{u\in W} g_u \delta_u)=
\sum_{w,u\in W} f_w w(g_u) \delta_{w u}.
$$

Define $y_i\in Q^{KK}_{t_1,t_2}(W)$ ($i=1,\ldots,r$) by
$$y_i:=A_i\delta_i+B_i
\text{\; where \;} 
A_i:=\frac{t_1+t_2 e^{-\alpha_i}}{1-e^{\alpha_i}},
B_i:=\frac{t_1+t_2}{1-e^{-\alpha_i}}.$$ 

\begin{proposition} We have the following equations.

(1) $(y_i-t_1)(y_i-t_2)=0$ \text{ for } $i=1,\ldots, r$.

(2) $\underbrace{y_i y_j\cdots}_{m_{i,j}}=\underbrace{y_j  y_i\cdots}_{m_{i,j}}$, where $m_{i,j}$ is the order of $s_i s_j$.
\end{proposition}
\begin{proof}
These equations can be shown by direct calculations.
\end{proof}

By this proposition we can define $y_w:=y_{i_1}\cdots y_{i_\ell}$ 
for a reduced expression $w=s_{i_1}\cdots s_{i_\ell}$.
These $\{y_w\}_{w\in W}$ become a $Q_{t_1,t_2}(\Lambda)$-basis of $Q^{KK}_{t_1,t_2}(W)$.

\begin{rem}
This operator $y_i$ can be seen as a generic Demazure-Lusztig operator.
When $t_1=-1, t_2=q$, it becomes $y^q_{s_i}$ in Kumar's book\cite{Kum}(12.2.E(9)).
We can also set $A_i$ which satisfies 
$$A_i A_{-i}=\frac{(t_1+t_2 e^{-\alpha_i})(t_1+t_2 e^{\alpha_i})}{(1- e^{\alpha_i})(1-e^{-\alpha_i})}.$$
For example, if we set $A_i=\frac{t_1+t_2 e^{\alpha_i}}{1-e^{\alpha_i}}$
and  $t_1=q,t_2=-1$ and replace $\alpha_i$ by $-\alpha_i$, it becomes
Lusztig's $T_{s_i}$ \cite{Lu1}.
If we set $A_i=-\frac{t_1+t_2 e^{\alpha_i}}{1-e^{-\alpha_i}}$
and $t_1=-1, t_2=v$ and replace $\alpha_i$ by $-\alpha_i$,
it becomes $\mathcal T_i$ in \cite{BBL}.

\end{rem}

We can define  a $Q_{t_1,t_2}(\Lambda)$-module isomorphism 
$\Phi:Q^{KK}_{t_1,t_2}(W)\to H^{Q(\Lambda)}_{t_1,t_2}(W)$
 by $\Phi(y_w)= h_w$.
Let $\Delta_{s_i}:=A_i \delta_i$. 
Define $A(w):=\displaystyle\prod_{\beta\in R(w)}\frac{t_1+t_2 e^{-\beta}}{1-e^{\beta}}$
and
$\Delta_w:=A(w)\delta_w$.
Then it becomes that
$\Delta_{s_{i_1}}\cdots \Delta_{s_{i_\ell}}
=A(w) \delta_w=\Delta_w$.
In particular, $\Delta_{s_i}$'s satisfy the braid relations.
We can show below by induction on length $\ell(w)$.

\begin{theorem}

 For $w\in W$, we have
$$\Phi(\Delta_w)=Y_w.$$

\end{theorem}

\begin{proof}
If
$w=s_i$, $\Delta_{s_i}=A_i\delta_i=y_i-B_i$. Therefore 
$\Phi(\Delta_{s_i})=h_i-B_i=h_i+\frac{t_1+t_2}{E(\alpha_i)}=Y_{s_i}$.
If $s_i w>w$, by induction hypothesis we can assume
$\Phi(\Delta_w)=Y_w=\displaystyle\sum_{u\leq w} p(u,w) h_u$.
As $\Phi$ is a $Q_{t_1,t_2}(\Lambda)$-isomorphism,
it follows that
$\Delta_w=\displaystyle\sum_{u\leq w} p(u,w) y_u$.
Then
$\Delta_{s_i w}=\Delta_{s_i}\Delta_w=
A_i\delta_i \displaystyle\sum_{u\leq w} p(u,w) y_u
=\displaystyle\sum_{u\leq w}s_i[ p(u,w)] A_i\delta_i  y_u
=\displaystyle\sum_{u\leq w}s_i[ p(u,w)] (y_i-B_i) y_u
=\displaystyle\sum_{u\leq s_i w} p(u,s_iw) y_u
$.
We used the recurrence relation (Proposition 5) for the last equality.
Therefore $\Phi(\Delta_{s_i w})=\displaystyle\sum_{u\leq s_i w} p(u,s_i w) h_u
=Y_{s_i w}.$ The theorem is proved.
\end{proof}

\begin{corollary}(Explicit formula for $\tilde{p}(w,v)$)

Let $v=s_{i_1}\cdots s_{i_\ell}$ be a reduced expression. Then we have

$$
\tilde{p}(w,v)
=
\frac{1}{A(w)}
\sum_{\epsilon=(\epsilon_1,\cdots,\epsilon_\ell)\in \{0,1\}^\ell, 
s_{i_1}^{\epsilon_1}\cdots s_{i_\ell}^{\epsilon_\ell}=w}
\prod_{j=1}^{\ell}C_j(\epsilon)
$$
where
 for $\epsilon=(\epsilon_1,\cdots,\epsilon_\ell)\in \{0,1\}^\ell$,
$C_j(\epsilon)
:=s_{i_1}^{\epsilon_1} s_{i_2}^{\epsilon_2}\cdots s_{i_{j-1}}^{\epsilon_{j-1}}
(\delta_{\epsilon_j,1}A_{i_j}+\delta_{\epsilon_j,0} B_{i_j})$.

\end{corollary}
\begin{proof}
Taking the inverse image of the map $\Phi$, the equality
$h_v=\sum_{w\leq v} \tilde{p}(w,v) Y_w$ becomes
$$y_v=\sum_{w\leq v} \tilde{p}(w,v) \Delta_w=\sum_{w\leq v} \tilde{p}(w,v) A(w)\delta_w.$$

As $v=s_{i_1}\cdots s_{i_\ell}$ is a reduced expression,
$y_v=y_{s_{i_1}}\cdots y_{s_{i_\ell}}
=(A_{i_i}\delta_{i_1}+B_{i_1}\delta_e)\cdots (A_{i_\ell}\delta_{i_\ell}+B_{i_\ell}\delta_e)$.
By expanding this we get the formula.

\end{proof}

\begin{rem}
Using Theorem 1, we also have a closed form for $p(w,v)$.
We have another conjectural formula  for $p(w,v)$ using 
$\lambda$-chain  cf. \cite{Nar}.

\end{rem}

\begin{example}
Type $A_2$.  We use notation 
$A_{-1}=\star(A_{1})$, $B_{-1}=\star(B_{1})$, 
$B_{12}=\frac{t_1+t_2}{1-e^{-(\alpha_1+\alpha_2)}}$.

When $v=s_1 s_2 s_1$, $w=s_1$, then $\epsilon=(1,0,0),(0,0,1)$
and 

$$\tilde{p}(s_1, s_1 s_2 s_1)=(A_1 B_{12} B_{-1}+B_1 B_2 A_{1})/A_{1}
=B_{12}B_{-1}+B_1 B_2=B_2 B_{12}.$$

When $v=s_1 s_2 s_1$, $w=s_2$, then $\epsilon=(0,1,0)$
and 

$$\tilde{p}(s_2, s_1 s_2 s_1)=(B_1 A_2 B_{12}) /A_{2}=B_1 B_{12}.$$

When $v=s_1 s_2 s_1$, $w=e$, then $\epsilon=(0,0,0),(1,0,1)$
and 

$$\tilde{p}(e, s_1 s_2 s_1)=B_1 B_2 B_1+A_1 B_{12} A_{-1}.$$
 
\end{example}

\section{Casselman's problem}

In his paper \cite{Cas} B. Casselman gave a problem concerning transition coefficients between
two bases in the space of Iwahori fixed vectors of a principal series representation of a $p$-adic group.
We relate the problem with the Yang-Baxter basis and give an answer to the problem.
\subsection{Principal series representations of $p$-adic group and Iwahori fixed vector}
We follow the notations of M.Reeder \cite{Re1,Re2}.
Let $G$ be a connected reductive $p$-adic group over a non-archimedian local field $F$.
For simplicity we restrict to the case of split semisimple $G$.
Associated to $F$, there is the ring of integer $\mathcal O$, the prime ideal $\mathfrak p$
with a generator $\varpi$, and the residue field with
$q=|\mathcal O/\mathfrak p|$ elements.
Let $P$ be a  minimal parabolic subgroup (Borel) of $G$, and
$A$ be the  maximal split torus of $P$ so that $A\simeq (F^{*})^r$ 
where $r$ is the rank of $G$.
 For an unramified quasi-character $\tau$ of  $A$,
 i.e.   a group homomorphism  $\tau:A\to \C^{*}$ which is
 trivial on $A_0=A\cap K$, where 
 $K=G(\mathcal O)$ is a maximal compact subgtoup of $G$.
 Let $T=\C^{*}\otimes X^{*}(A)$
be the complex torus dual to $A$, where $X^{*}(A)$ is the group of rational characters on $A$, i.e. 
$X^{*}(A)=\{\lambda:A\to F^{*}, \text{ algebraic group homomorphism}\}$.
We have a pairing  $<,>:A/A_0\times T\to \C^{*}$ given by
$<a, z\otimes \lambda>=z^{{\rm val}(\lambda(a))}$.
This gives an identification $T\simeq X^{nr}(A)$ of  $T$ with the set of unramified quasi-characters on $A$
(cf. \cite{Bum} Exercise 18,19).

Let $\Delta\subset X^{*}(A)$ be the set of roots of $A$ in $G$, $\Delta^{+}$ be the set of
positive roots corresponding to $P$ 
and $\Sigma\subset \Delta^{+}$ be the set of simple roots
.
For a root $\alpha\in \Delta$, we define $e_\alpha\in X^{*}(T)$ by
$$e_{\alpha}(\tau)=<h_\alpha(\varpi),\tau>$$
for $\tau\in T$ where $h_\alpha: F^{*}\to A$ is the one parameter subgroup (coroot)
corresponding to $\alpha$.

\begin{rem}
As the definition shows, $e_\alpha$ is defined using the coroot $\alpha^{*}=h_\alpha$.
So it should be parametrized by $\alpha^{*}$, but for convenience we follow the notation
of \cite{Re1}.
Later we will identify $e_\alpha (\alpha\in \Delta=R^{*})$ with $e^\alpha  (\alpha\in R=\Delta^{*})$
by the map $*:\Delta\to R$ of root data.
\end{rem}

$W$ acts on right of $X^{nr}(A)$
so that
$\tau^w(a)=\tau(w a w^{-1})$ for $a\in A$, $\tau\in T$ and $w\in W$.
The action of $W$ on $X^{*}(T)$ is given by
$(w e_\alpha)(\tau)=e_{w \alpha}(\tau)=e_{\alpha}(\tau^w)$
for $\alpha \in \Delta$, $\tau\in T$ and $w\in W$.

The principal series representation $I(\tau)$ of $G$ associated to a unramified quasicharacter $\tau$ of $A$ 
is defined as follows.
As a vector space over $\C$ it consists of locally constant functions on $G$ with values in $\C$ which satisfy the
left relative invariance properties with respect to $P$ where $\tau$ is extended to $P$ with trivial value on the unipotent radical $N$
of $P=AN$.
$$I(\tau):={\rm Ind}_{P}^{G} (\tau)=\{f:G\to \C 
\text{ loc. const. function }| f(pg)=\tau \delta^{1/2}(p) f(g) \;
\text{ for } \forall p\in P, \forall g\in G \}.$$
Here $\delta$ is the modulus of $P$.
The action of $G$ on $I(\tau)$ is defined by right translation, i.e.
for $g\in G$ and $f\in I(\tau)$,
$(\pi(g) f)(x)=f(xg)$.

Let 
$B$ be the Iwahori subgroup
which is the inverse image $\pi^{-1}(P(\F_q))$ of the  Borel subgroup $P(\F_q)$ of $G(\F_q)$ by the projection
$\pi:G(\O)\to G(\F_q)$. Then we define $I(\tau)^B$ to be the space of Iwahori fixed vectors in $I(\tau)$,
i.e.
$$I(\tau)^B:=\{f\in I(\tau) \;|\; f(g b)=f(g) \text{ for } \forall b\in B, \forall g\in G\}.$$
This space has a natural basis $\{\varphi^\tau_w\}_{w\in W}$.
$\varphi^\tau_w\in I(\tau)^B$ is supported on $PwB$ and satisfies
$$\varphi^\tau_w (p w b)=\tau \delta^{1/2}(p)\; \text{ for } \forall p\in P, \forall b\in B.$$

\subsection{Intertwiner and Casselman's basis}
From now on we always assume that {$\tau$ is regular} i.e. the stabilizer 
$W_{\tau}=\{w\in W\;|\; \tau^w=\tau\}$
is trivial.
The intertwining operator
$\A^\tau_w: I(\tau)\to I(\tau^w)$ is
defined by
$$\A^\tau_w(f) (g):=
\displaystyle\int_{N_w} f(w n g) dn$$
where
$N_w:=N\cap w^{-1} N_{-} w$, 
with $N_{-}$ being the unipotent radical of opposite parabolic $P_{-}$ 
which
corresponds to the negative roots $\Delta^{-}$. 
The integral is convergent when $|e_\alpha(\tau)| <1$ 
for all $\alpha\in \Delta^{+}$ such that $w\alpha\in \Delta^{-}$
(cf. \cite{Bum} Proposition 63),
and may be meromorphically continued to all $\tau$.
It has the property that for $x,y\in W$ with $\ell(xy)=\ell(x)+\ell(y)$,
then
$$\A^{\tau^x}_y \A^{\tau}_{x}=\A^{\tau}_{xy}.$$

The Casselman's basis $\{f^\tau_w\}_{w\in W}$ of $I(\tau)^B$ is 
defined as follows.
$f^\tau_w\in I(\tau)^B$ and 

$$\A^\tau_y f^\tau_w(1)=\begin{cases}
1& \text{ if } y=w\\
0& \text{ if } y\neq w.\\
\end{cases}
$$

M.Reeder characterizes this using the action of affine Hecke algebra (cf. \cite{Re2} Section 2).
The affine Hecke algebra $\H=\H(G,B)$ is the convolution algebra of
$B$ bi-invariant locally constant functions on $G$ with values in $\C$.
By the theorem of Iwahori-Matsumoto it can be described by generators and relations.
The basis $\{T_w\}_{w\in \widetilde{W}_{aff}}$ consists of 
characteristic functions $T_w:=ch_{BwB}$
of double coset $BwB$.
 Let $\H_{W}$ be the Hecke algebra of the finite Weyl group $W$ generated by the simple reflections 
 $s_\alpha$ for simple roots $\alpha\in \Sigma$.
 As a vector space $\H$ is the tensor product of two subalgebras
 $\H=\Theta\otimes \H_W$.
 The subalgebra $\Theta$ is commutative
and  isomorphic to the coordinate ring of the complex torus $T$
with a basis $\{\theta_a\;|\;a\in A/A_0\}$, where
 $\theta_a$ is defined as follows  (cf. \cite{Lu2}).
 Define $A^{-}:=\{a\in A\;|\; |\alpha(a)|_F\leq 1 \;\forall \alpha\in \Sigma\}$.
 For $a\in A$, choose $a_1,a_2\in A^{-}$ such that $a=a_1 a_2^{-1}$. Then
 $\theta_a=q^{(\ell(a_1)-\ell(a_2))/2} T_{a_1} T_{a_2}^{-1}$
 where for $x\in G$, $\ell(x)$ is the length function defined by $q^{\ell(x)}=[BxB:B]$
 and $T_x\in\H$ is the characteristic function of $BxB$.

By Lemma (4.1) of \cite{Re1}, there exists a unique
$f^{\tau}_w\in I(\tau)_{w}\cap I(\tau)^B$ for each $w\in W$
such that 

$(1) f^\tau_w(w)=1$ and 

$(2) \pi(\theta_a) f^\tau_w=\tau^{w}(a) f^\tau_w$ for all $a\in A$.

\noindent
Here $I(\tau)_w:=\{f\in I(\tau)\;|\; \text{ support of  $f$ is contained in } \bigcup_{x\geq w}PxP\}$.



\vspace{0.5cm}

\subsection{Transition coefficients}

Let
$$f^\tau_w =\sum_{w\leq v} a_{w,v}(\tau) \varphi^\tau_{v}$$
and
$$\varphi^\tau_w =\sum_{w\leq v} b_{w,v}(\tau) f^\tau_{v}.$$

\noindent
The Casselman's problem is to find an explicit formula for $a_{w,v}(\tau) $ and $b_{w,v}(\tau)$.

To relate the results in Sections 2 and 3 with the Casselman's problem,
in this subsection we specialize the parameters $t_1=-q^{-1}$, $t_2=1$
and take tensor product with the complex field $\C$.
For example, the Yang-Baxter basis
$Y_w$ will become 
a $Q_{t_1,t_2}(\Lambda)\otimes \C$ basis
 in $H^{Q(\Lambda)}_{t_1,t_2}(W)_{\C}=H^{Q(\Lambda)}_{t_1,t_2}(W)\otimes \C$.
The generic Demazure-Lusztig operator defined in Section 3 will become
$$y_i:=A_i\delta_i+B_i
\text{\; where \;} 
A_i:=\frac{-q^{-1}+ e^{-\alpha_i}}{1-e^{\alpha_i}},
B_i:=\frac{-q^{-1}+1}{1-e^{-\alpha_i}}.$$ 

\noindent
Then $(y_i+q^{-1})(y_i-1)=0$.

\begin{theorem} We identify $e^\alpha$ with $e_{\alpha}$ (cf. Remark 4). Then,

$$a_{w,v}(\tau)=\tilde{p}(w,v)(\tau)|_{t_1=-q^{-1},t_2=1}$$
$$b_{w,v}(\tau)={p}(w,v)(\tau)|_{t_1=-q^{-1},t_2=1}.$$
\end{theorem}
\begin{proof}
$b_{w,v}$'s satisfy the same recurrence relation (Proposition 5 with $t_1=-q^{-1},t_2=1$) 
as $p(w, v)$'s (cf. \cite{Re2} Proposition (2.2)).
The initial condition  $b_{w,w}=p(w,w)=1$ leads to the second equation.
The first equation then also holds. Note that  the $b_{y,w}$ in \cite{Re2} is our $b_{w,y}$.

\end{proof}

\begin{rem}
There is also a direct proof that does not use recurrence relation
cf. \cite{NN}.

 \end{rem}
 
 \begin{corollary}
 We have a closed formula for $a_{w,v}(\tau)$ and $b_{w,v}(\tau)$ by
 Corollary 2 and Theorem1.
 \end{corollary}

 \begin{corollary} For $v\in W$, we have
 $$\sum_{w\leq v} b_{w,v}=
 \prod_{\beta\in R(v)} 
 \frac{1-q^{-1} e^\beta}{1-e^\beta},$$
 and
  $$\sum_{w\leq v} b_{w,v} (-q^{-1})^{\ell(w)}=
  \prod_{\beta\in R(v)} \frac{1-q^{-1}}{1-e^\beta}.$$
 \end{corollary}
 \begin{proof}
 When $t_1=-q^{-1},t_2=1$, 
 we can specialize $h_i$ to 1 and we get the first equation
 from the definition of $Y_v$, since $1+\frac{(1-q^{-1})e^\beta}{1-e^\beta}= \frac{1-q^{-1} e^\beta}{1-e^\beta}$.
 We can also specialize $h_i$ to $-q^{-1}$ and
 $-q^{-1}+\frac{(1-q^{-1})e^\beta}{1-e^\beta}= \frac{1-q^{-1}}{1-e^\beta}$ gives the second equation.
 \end{proof}
 
 \begin{rem}
 The left hand side of the first equation in Corollary 4 is $m(e,v^{-1})$ in \cite{BN}.
 So this gives another proof of Theorem 1.4 in \cite{BN}.
 \end{rem}

\subsection{Whittaker function}

 M.Reeder \cite{Re2} specified a formula for the Whittaker function
  $\mathcal W_\tau(f^\tau_w)$ and
using $b_{w,v}$, he got a formula for $\mathcal W_\tau(\varphi^\tau_w)$.
 For $a\in A$,
let $\lambda_a\in X^{*}(T)$ be 
$$\lambda_a(z\otimes \mu)=z^{val(\mu(a))} 
\text{ for } z\in \C^{*}, \mu\in X^{*}(A). $$

Formally the result of M.Reeder \cite{Re2}  Corollary $(3.2)$ 
 is written as follows. For $w\in  W$ and $a\in A^{-}$,

$$\W(\varphi_w)(a)=
\delta^{1/2}(a) \sum_{w\leq y}
b_{w,y}\; y \left[ \lambda_a \prod_{\beta\in R^{+}-R(y)} \frac{1-q^{-1} e^{\beta}}{1-e^{-\beta}} \right]\in \C[T].$$
 Then using Corollary 3, we have an explicit formula of $\W(\varphi_w)(a)$.

\subsection{Relation with Bump-Nakasuji's work}
Now we explain the relation between this paper and Bump-Nakasuji \cite{BN}.
First of all, the notational conventions are slightly different.
Especially in the published \cite{BN} the natural base and intertwiner are differently parametrized.
 The natural basis $\phi_w$ in \cite{BN} is our $\varphi_{w^{-1}}$.
The intertwiner $M_{w}$ in \cite{BN} is our $\mathcal A_{w^{-1}}$ so that if $\ell(w_1 w_2)=\ell(w_1)+\ell(w_2)$,
$M_{w_1 w_2}=M_{w_1}\circ M_{w_2}$ while $\mathcal A_{w_1 w_2}=\mathcal A_{w_2} \mathcal A_{w_1}$.

In the paper \cite{BN}, another basis $\{\psi_w\}_{w\in W}$ for 
 the space $I(\tau)^B$ was defined and comparerd with the Casselman's
basis. They defined $\psi_w:=\sum_{v\geq w} \varphi_v$ and
expand this as 
$\psi_w=\sum_{v\geq w} m(w,v) f_v$ and  conversely
$f_w=\sum_{v\geq w} \tilde{m}(w,v) \psi_v$
. 
They observed that the transition coefficients $m(w,v)$ and $\tilde{m}(w,v)$
factor under certain condition.
Let 
$S(w,v):=\{ \alpha\in R^{+} |
w\leq s_\alpha v<v\}$
and
$S'(w,v):=\{ \alpha\in R^{+}  |
w< s_\alpha w\leq v\}$.
Then the statements of the conjectures are as follows.
\begin{conjecture}(\cite{BN} Conjecture 1.2)
Assume that the root system $R$ is simply-laced. 
Suppose $w\leq v$ and  $|S(w,v)|=\ell(v)-\ell(w)$, then
$$m(w,v)=
\displaystyle\prod_{\alpha\in S(w,v)}
\frac{1-q^{-1} z^\alpha}{1-z^\alpha}.$$
\end{conjecture}

\begin{conjecture}(\cite{BN} Conjecture 1.3)
Assume that the root system $R$ is simply-laced. 
Suppose $w\leq v$ and  $|S'(w,v)|=\ell(v)-\ell(w)$, then
$$\tilde{m}(w,v)=
(-1)^{\ell(v)-\ell(w)}\displaystyle\prod_{\alpha\in S'(w,v)}
\frac{1-q^{-1} z^\alpha}{1-z^\alpha}.$$
\end{conjecture}

\begin{proposition}
Conjecture 1.2 and Conjecture 1.3 in \cite{BN} are equivalent.
\end{proposition}

\begin{proof}
We can show $m(w,v)=\displaystyle\sum_{w\leq z\leq v} p(z,v)$ and
$\tilde{m}(w,v)=\displaystyle
\sum_{w\leq z\leq v} (-1)^{\ell(v)-\ell(z)}\tilde{p}(w,z)$.
Then it follows by the Theorem 1 that
$\tilde{m}(w,v)=(-1)^{\ell(v)-\ell(w)} m(v w_0, w w_0)$.
As $S'(w,v)=S(v w_0, w w_0)$ we get 
the desiered conclusion.
\end{proof}

{\bf Acknowledgements}

We would like to thank Shin-ichi Kato for valuable comments on the
first version of this paper.
We also would like to thank the referees for careful reading and point
out the significantly related paper \cite{Che}.
This work was supported in part by JSPS Research Fellowship for Young
Scientists and by Grant-in-Aid for Scientific Research.


\end{document}